\documentclass[a4paper,12pt]{amsart}
\usepackage{latexsym,amscd,amssymb,amsmath,mathrsfs,comment,url}
\makeatletter
\theoremstyle{plain}
  \newtheorem{thm}{Theorem}[section]
  \newtheorem{prop}[thm]{Proposition}
  \newtheorem{lem}[thm]{Lemma}
  \newtheorem{cor}[thm]{Corollary}
  \newtheorem{conj}[thm]{Conjecture}
\theoremstyle{definition}
  \newtheorem{dfn}[thm]{Definition}
  \newtheorem{exmp}[thm]{Example}

  \newtheorem{rem}[thm]{Remark}
\let\opn\operatorname 
\newcommand\@bothmode[1]{\ifmmode #1\else $#1$ \fi}
\numberwithin{equation}{section}
\renewcommand\theenumi{\@roman\c@enumi}
\renewcommand{\@biblabel}[1]{#1.}
\newcommand\chdir[1]{%
\let\pwd\relax%
\def\pwd{#1}}
\newcommand\inputs[1]{\@for\@tempmember:=#1\do{\input{\pwd /\@tempmember}}}
\DeclareFontFamily{U}{futm}{}
\DeclareFontShape{U}{futm}{m}{n}{
  <-> s * [.92] fourier-bb
  }{}
\DeclareMathAlphabet{\math@bb}{U}{futm}{m}{n}
\let\mathbb\math@bb\relax
\DeclareMathAlphabet{\math@cal}{OMS}{cmsy}{m}{n}
\let\mathcal\math@cal\relax
\newcommand\NN{\mathbb{N}} 
\newcommand\ZZ{\mathbb{Z}} 
\newcommand\kk{\mathbb{k}} 
\newcommand\too{\longrightarrow}
\newcommand\ba{\mathbf{a}}
\newcommand\bb{\mathbf{b}}
\newcommand\bc{\mathbf{c}}
\newcommand\bd{\mathbf{d}}
\newcommand\be{\mathbf{e}}
\newcommand\br{\mathbf{r}}
\newcommand\one{\mathbf{1}}
\newcommand\zero{\mathbf{0}}
\newcommand\bax{\bar{x}}
\let\gb\beta

\let\gs\sigma

\let\gD\Delta

\let\@tempar\relax 
\def\@seton^#1{\overset{#1}{\@tempar}}
\def\@setbeneath_#1{\underset{#1}{\@tempar}}
\newcommand\defar[2]{\@xp\def\csname #1\endcsname{
    \def\@tempar{#2}\@ifnextchar^{\@seton}{
        \@ifnextchar_{\@setbeneath}{\@tempar}
    }}}
\defar{longto}{\longrightarrow} 
\defar{epito}{\twoheadrightarrow} 
\defar{monoto}{\rightarrowtail} 
\defar{embto}{\hookrightarrow} 
\newcommand\imply{\@bothmode{\Rightarrow}} 
\newcommand\Imply{\@bothmode{\Longrightarrow}} 
\renewcommand\iff{\@bothmode\Leftrightarrow} 
\newcommand\Iff{\@bothmode\Longleftrightarrow}
\newcommand\get{\@bothmode{\Leftarrow}} 
\newcommand\Get{\@bothmode{\Longleftarrow}} 
\newcommand\bra[1]{[#1]} 
\newcommand\mbra[1]{\left\{ #1 \right\}} 
\newcommand\set[2]{\mbra{\,#1\ \left|\ #2\, \right.}} 

\let\Dsum\bigoplus 
\newcommand\defopn[1]{\@xp\def\csname #1\endcsname{\opn{\csname the#1\endcsname\relax}}}
\defopn{Ker}

\defopn{Im} 
\newcommand\ann{\opn{ann}}
\newcommand\depth{\opn{depth}}
\newcommand\sdepth{\opn{sdepth}}
\newcommand\supp{\opn{supp}}

\newcommand\sd{\opn{sd}}
\newcommand\qsd{\opn{qsd}}
\newcommand\reg{\opn{reg}}
\newcommand\hreg{\opn{hreg}}
\newcommand\sreg{\opn{supp. \!reg}}
\newcommand\shreg{\widetilde{h}\opn{-reg}}
\newcommand\seq[2][n]{{#2}_1,\dots ,{#2}_{#1}}
\newcommand\poly[2][\kk]{#1\bra{#2}}
\def\@polys[#1][#2]#3{\poly[#1]{\seq[#2]#3}}
\def\@polysx[#1]{\@ifnextchar[{\@polys[#1]}{\@polys[\kk][#1]}}
\def\polys{\@ifnextchar[{\@polysx}{\@polys[\kk][n]}}
\newcommand\fm{\mathfrak{m}}
\renewcommand\mod{\opn{mod}}
\newcommand\Hom{\opn{Hom}}

\newcommand\Ext{\opn{Ext}}

\newcommand\sA{\mathscr{A}}
\newcommand\op{\mathsf{op}}

\newcommand\cD{\mathcal{D}}
\makeatother
\title[Alexander duality and Stanley depth]{Alexander duality and Stanley depth 
of multigraded modules}
\author{Ryota Okazaki}
\address{Department of Pure and Applied Mathematics, 
Graduate School of Information Science and Technology,
Osaka University, Toyonaka, Osaka 
560-0043, Japan}
\email{u574021d@ecs.cmc.osaka-u.ac.jp}
\author{Kohji Yanagawa}
\thanks{The first author is partially supported by JST, CREST. 
The second author is partially supported by Grant-in-Aid for Scientific Research (c) (no.19540028).}
\address{Department of Mathematics, Kansai University,
Suita 564-8680, Japan}
\email{yanagawa@ipcku.kansai-u.ac.jp}
\begin{document}

\maketitle

\begin{abstract}
We apply Miller's theory on multigraded modules over a polynomial ring  
to the study of the Stanley depth 
of these modules. Several tools for Stanley's conjecture are developed, 
and a few partial answers are given.  
For example, we show that taking the Alexander duality twice (but with 
different ``centers") is useful for this subject.  Generalizing a result of Apel, 
we prove that Stanley's conjecture holds for the quotient by a cogeneric monomial ideal.
\end{abstract}

\section{Introduction}
Let $S=\kk[x_1, \ldots, x_n]$ be a polynomial ring over a field $\kk$. 
We regard it as a $\ZZ^n$-graded ring in the natural way. 
Let $\mod_{\ZZ^n} S$ be the category of finitely generated $\ZZ^n$-graded 
$S$-modules and degree preserving $S$-homomorphisms between them.   
We say $M = \bigoplus_{\ba \in \ZZ^n} M_\ba \in \mod_{\ZZ^n} S$ is $\NN^n$-graded if 
$M_\ba = 0$ for all $\ba \not \in \NN^n$. 
Let $\mod_{\NN^n} S$ denote the full subcategory of $\mod_{\ZZ^n} S$ 
consisting of $\NN^n$-graded modules.

For a subset $Z \subset \{ x_1, \ldots, x_n \}$, $\kk[Z]$ denotes 
the $\kk$-subalgebra of $S$ generated by all $x_i \in Z$. 
Clearly, $\kk[Z]$ is a polynomial ring with $\dim \kk[Z] = \# Z$. 
Let $M \in \mod_{\ZZ^n} S$.  We say the $\kk[Z]$-submodule $m \, \kk[Z]$ of $M$ 
generated by a homogeneous element $m \in M_\ba$ is a {\it Stanley space}, if it is 
$\kk[Z]$-free.  Note that $m \, \kk[Z]$ is a Stanley space if and only if 
$\ann(m) \subset (x_i \mid x_i \not \in Z)$. 
A {\it Stanley decomposition} $\cD$ of $M$ is a presentation of 
$M$ as a finite direct sum of Stanley spaces. That is, 
$$\cD:\bigoplus_{i=1}^s m_i \, \kk[Z_i] = M$$ 
as $\ZZ^n$-graded $\kk$-vector spaces, where each $m_i \, \kk[Z_i]$ 
is a Stanley space. 

Let $\sd(M)$ be the set of Stanley decompositions of $M$. 
For all $0 \ne M \in \mod_{\ZZ^n} S$, we have $\sd (M) \ne \emptyset$. 
For $\cD=\bigoplus_{i=1}^s m_i \, \kk[Z_i] \in \sd (M)$, we set 
$$\sdepth (\cD) := \min\set{\# Z_i}{i = 1,\dots ,s},$$
and call it the  {\em Stanley depth} of $\cD$. 
The Stanley depth of $M$ is defined by 
$$
\sdepth (M) := \max\set{\sdepth \cD}{\cD \in \sd (M)}.
$$

While it is obvious that $\sdepth M \leq \dim_S M$, 
this invariant behaves somewhat strangely. For example, if $I$ is a 
complete intersection monomial ideal of codimension $c$ then we have 
$\sdepth (S/I) = n-c$ but $\sdepth I = n - \lfloor \frac{c}{2} \rfloor$ 
as shown in \cite{She}. The following is a special case of the conjecture raised in \cite{St82}.
 
\begin{conj}
[Stanley]\label{Stanley conj}
Assume $\kk$ is infinite. For any $M \in \mod_{\ZZ^n} S$, we have
$$
\sdepth M \ge \depth M.
$$
(If $M = I/J$ for some monomial ideals $I,J$ of $S$ with $I \supset J$, 
then the assumption that $\kk$ is infinite is superfluous.)
\end{conj}

After the works of Apel's (\cite{A1,A2}), the conjecture has been intensely studied.  
(See for example \cite{HSZ,HVZ,She, SoJ}. Here we listed papers directly related to the 
present paper, and there are many other interesting works.) 
However the conjecture is still widely open. 
No relation between $\sdepth I$ and  $\sdepth (S/I)$ is known in the general case, 
hence the conjecture for $I$ itself and that for $S/I$ are different stories. 

In \cite{M}, Miller introduced the notion of {\it positively $\ba$-determined} $S$-modules 
for each $\ba \in \NN^n$. These modules form the full subcategory $\mod_\ba S$ of $\mod_{\NN^n} S$,  
which admits the {\it Alexander duality functor} $\sA_\ba : \mod_\ba S \to (\mod_\ba S)^\op$. 
Any $M \in \mod_{\NN^n} S$ is positively $\ba$-determined for sufficiently large $\ba \in \NN^n$, 
and $\sdepth M$ is attained by a positively $\ba$-determined Stanley decomposition in this case. 
Hence we can study the Stanley depth in Miller's context. 
For $\one := (1,1, \ldots, 1) \in \NN^n$,  positively $\one$-determined modules are nothing 
other than {\it squarefree modules} introduced in \cite{Y}. 

For a squarefree module $M$ and a squarefree (i.e., positively $\one$-determined) Stanley decomposition $\cD$ of $M$, 
Soleyman Jahan~\cite{SoJ} defined the Alexander dual $\sA_\one(\cD) \in \sd(\sA_\one(M))$ of $\cD$. 
However, it is impossible to generalize his construction to $\mod_\ba S$ and $\sA_\ba$  directly. 
So we will introduce the notion of {\it quasi Stanley decompositions}. 
Let $\qsd (M)$ (resp. $\qsd_\ba (M)$) be the set of  (resp.  positively $\ba$-determined) 
quasi Stanley decompositions of $M \in \mod_\ba S$. 
Then $\sd(M) \subset \qsd (M) = \bigcup_{\ba \in \NN^n} \qsd_\ba (M)$ and 
$\sdepth M$ can be computed also by $\qsd_\ba (M)$ or  $\qsd (M)$.   
Moreover, the Alexander duality $\sA_\ba$ 
gives a bijection from $\qsd_\ba(M)$ to $\qsd_\ba(\sA_\ba (M))$.

Using $\qsd (M)$, we can define a new invariant $\shreg(M)$. 
As an analog of Miller's equation 
$$ \sreg(M)+ \depth (\sA_\ba(M))=n$$
(the {\it support regularity} $\sreg (M)$ of $M$ is introduced also by Miller), 
we have  $$ \shreg(M)+\sdepth (\sA_\ba(M))=n.$$ 
Hence Stanley's conjecture 
(Conjecture~\ref{Stanley conj}) is equivalent to the conjecture that 
$\shreg (M) \leq \sreg (M)$ for all $M \in \mod_{\NN^n} S$.   
If $M$ is squarefree, then $\sreg(M)$ equals the usual (Castelnuovo-Mumford) 
regularity of $M$, and $\shreg M$ equals $\hreg M$ defined in Soleyman Jahan~\cite{SoJ}. 
Hence our observation is a generalization of that in \cite{SoJ}. 

For $l \in \NN$, we define the $l^{\rm th}$ {\it skeleton} $M^{\leq l}$ of $M \in \mod_\ba S$. 
The prototype of this idea is the skeletons of simplicial complexes and their Stanley-Reisner rings.   
Hence $M^{\leq l}$ is a quotient module of $M$ with $\dim_S M^{\leq l} \le l$. 
Using this notion, in Theorem~\ref{skeleton CM}, we show that Stanley's conjecture 
 holds for all $M \in \mod_{\ZZ^n} S$ 
if and only if it holds for all $M \in \mod_{\ZZ^n} S$ which are Cohen-Macaulay. 
The ideal version of this result has been obtained by Herzog  et al. \cite{HSZ}. 

For $\ba, \bb \in \NN^n$, $(-)^{\triangleleft \bb}$ denotes the composition 
$\sA_{\ba+\bb} \circ \sA_\ba: \mod_\ba S \to \mod_{\ba+\bb} S$   
(more precisely, the composition of $\sA_\ba: \mod_\ba S \to (\mod_\ba S)^\op$, 
the natural inclusion $(\mod_\ba S)^\op \embto (\mod_{\ba+\bb} S)^\op$, and 
$\sA_{\ba+\bb}: (\mod_{\ba+\bb} S)^\op \to \mod_{\ba+\bb} S$). 
For $M \in \mod_{\NN^n} S$, $M^{\triangleleft \bb}$ does not depend on the particular 
choice of $\ba$ with $M \in \mod_{\ba} S$. 
Since we have $\depth M = \depth M^{\triangleleft \bb}$ and $\sdepth M = \sdepth M^{\triangleleft \bb}$, 
Stanley's conjecture holds for $M$ if and only if it holds for $M^{\triangleleft \bb}$. 

{\it Generic} and {\it cogeneric} monomial ideals are interesting combinatorial 
classes introduced in \cite{BPS,Str}. Apel (\cite{A1,A2}) showed that 
if  a monomial ideal $I$ is generic 
then Stanley's conjecture  holds for $I$ itself and $S/I$.  
In Theorem~\ref{cogeneric}, we show that if $I$ is cogeneric then the conjecture holds for $S/I$.   
Under the additional assumption that $S/I$ is Cohen-Macaulay, this result has been proved 
in \cite{A2}. Roughly speaking, our proof reduces the assertion to the Cohen-Macaulay case (\cite{A2}) 
using techniques developed in \S\S 2--5 of the present paper. However, since the skeletons of 
(co)generic monomial ideals are no longer (co)generic, we need modification. 
We also remark that more inclusive definitions of (co)generic monomial ideals were given in \cite{MSY}, 
and Apel used these new definitions. 
However our proof of Theorem~\ref{cogeneric} works only for the original definition. 

\medskip

Most results in \S\S 2--4 are taken from the thesis \cite{O} of the first author. 
The authors are grateful to Professor J\"urgen Herzog for helpful comments.

\section{Preliminaries}
Let $S$, $\mod_{\ZZ^n} S$ and $\mod_{\NN^n} S$ be as defined in the beginning of the previous section. 
The definitions of Stanley decompositions and the Stanley depth are also given there. 
Let $\sd(M)$ be the set of Stanley decompositions of $M \in \mod_{\ZZ^n} S$. 
In this paper, we sometimes regard $M \in \mod_{\ZZ^n} S$ as just a 
$\ZZ^n$-graded $\kk$-vector space without saying so explicitly. However, the context makes the meaning clear.

We start this section from the following lemma.

\begin{lem}\label{sec:ex seq and sdepth}
Given an exact sequence
$$
0 \longto L \stackrel{f}{\longto} M \stackrel{g}{\longto} N \longto 0
$$
in $\mod_{\ZZ^n} S$, it follows that
$$
\sdepth M \ge \min\{\, \sdepth L, \, \sdepth N \, \}.
$$

In particular, for a direct sum $M= \bigoplus_{i=1}^s M_i$ in $\mod_{\ZZ^n} S$, 
we have 
\begin{equation}\label{direct sum}
\sdepth M \geq \min \{ \, \sdepth M_i \mid 1 \leq i \leq s \, \}.
\end{equation}
\end{lem}

\begin{proof}
Let $\cD_1 = \bigoplus_{i=1}^s l_i \, \kk[Z_i] \in \sd(L)$ and 
$\cD_2 = \bigoplus_{i=1}^t n_i \, \kk[Z'_i] \in \sd(N)$ be Stanley decompositions 
attaining the Stanley depths of each modules. For $1\leq i \leq s$, set $m_i:= f(l_i) \in M$. 
For $s+1 \le i \le s+t$, take a homogeneous element $m_i \in M$ so that $g(m_i)= n_{i-s}$, and 
set $Z_i:= Z'_{i-s}$.
Then it is easy to see that each $m_i \, \kk[Z_i]$ is a Stanley space and 
$\sum_{i=1}^{s+t} m_i \, \kk[Z_i]= \bigoplus_{i=1}^{s+t} m_i \, \kk[Z_i] =M$. 
Hence  $\cD:= \bigoplus_{i=1}^{s+t} m_i \, \kk[Z_i]$ is a Stanley decomposition of $M$, and we have
$\sdepth M \ge \sdepth \cD = \min \{ \, \sdepth \cD_1, \sdepth \cD_2 \, \} = 
\min \{ \, \sdepth L, \sdepth N \,\}$. 
\end{proof}

\begin{rem}
The reader might think the equality holds in \eqref{direct sum} and the proof is easy. 
However, as far as the authors know, 
even whether $\sdepth (M \oplus S) = \sdepth M$ always holds or not is an open problem. 
\end{rem}

As usual, for $M \in \mod_{\ZZ^n} S$ and $\ba \in \ZZ^n$, $M(\ba) \in \mod_{\ZZ^n} S$ denotes the 
degree shift of $M$ with $M(\ba)_\bb=M_{\ba+\bb}$. For any $M \in \mod_{\ZZ^n} S$, 
there is some $\ba$ such that  $M(\ba) \in \mod_{\NN^n} S$.  
While Stanley's conjecture (Conjecture~\ref{Stanley conj}) concerns modules in $\mod_{\ZZ^n} S$, 
we can restrict our attention to $\mod_{\NN^n} S$ since the degree shift preserves 
both the usual and Stanley depths. 

Here, we introduce the convention on $\NN^n$ used throughout the paper. 
The $i^{\rm th}$ coordinate of $\ba \in \NN^n$ is denote by $a_i$. 
Let $\succeq$ be the order on $\NN^n$ defined by 
$\ba \succeq \bb \ \Iff \ a_i \ge b_i$ for all $i$. 
Clearly, $\zero:=(0,0, \ldots, 0) \in \NN^n$ is the smallest element.  
For $\ba,\bb \in \NN^n$,
let $\ba \vee \bb$, $\ba \land \bb$ be the elements of $\NN^n$ whose $i^{\text{th}}$-coordinates are
$\max\mbra{a_i,b_i}$, $\min\mbra{a_i,b_i}$ respectively.  If $\ba \succeq \bb$, we set
$\bra{\ba,\bb} := \set{\bc \in \NN^n}{\ba \preceq \bc \preceq \bb}.$

For $\ba, \bb \in \NN^n$, set
$$
\supp^\ba(\bb) := \set{i}{b_i \ge a_i}, \qquad \supp_X^\ba(\bb) := \set{x_i}{b_i \ge a_i}.
$$
For the simplicity, $\supp^\one(\bb) = \set{i}{b_i \ge 1}$ is denoted by $\supp(\bb)$, where 
$\one:=(1,1,\ldots, 1) \in \NN^n$. 
For a homogeneous element $0 \ne m \in M_\bb$ of $M=\bigoplus_{\ba \in \NN^n}M_\ba$, 
set $\deg(m)=\bb$, and $\supp^\ba(\deg(m))$ is simply denoted by $\supp^\ba(m)$. 
The monomial $\prod_{i=1}^n x_i^{a_i} \in S$ is denoted by $x^\ba$.

\begin{dfn}[Miller \cite{M}]
Let $\ba \in \NN^n$. 
We say a $\ZZ^n$-graded $S$-module $M$ is {\it positively $\ba$-determined}, if it is
finitely generated,  $\NN^n$-graded, and the multiplication 
map $M_\bb \ni m \longmapsto x_i m \in M_{\bb + \be_i}$   
is bijective for all $\bb \in \NN^n$ and all $i \in \supp^\ba(\bb)$. 
Here $\be_i \in \NN^n$ denotes the $i^{\rm th}$ unit vector. 
\end{dfn}

Let $\mod_\ba S$ be the full subcategory of $\mod_{\NN^n} S$ consisting 
of positively $\ba$-determined modules.  
If $\ba' \succeq \ba$, we have $\mod_{\ba'} S \supset \mod_\ba S$. 
Any $M \in \mod_{\NN^n} S$ is positively $\ba$-determined 
for sufficiently large $\ba \in \NN^n$. For example, a monomial ideal $I \subset S$ 
minimally generated by $x^{\ba_1}, x^{\ba_2}, \ldots, x^{\ba_r}$ is positively 
$\ba$-determined if and only if $\ba \succeq (\ba_1 \vee \ba_2 \vee \cdots \vee \ba_r)$.  

If $M \in \mod_\ba S$, the essential information of $M$ appears in the  
subspace $M_{[\zero, \ba]} := \bigoplus_{\bb \in [\zero, \ba]} M_\bb$. 
For example, we have 
\begin{align*}
\dim_S M &= \max \{ \, \# \supp^\ba (\bb) \mid \bb \in \NN^n, M_\bb \ne 0 \, \}\\
&= \max \{ \, \# \supp^\ba (\bb) \mid \bb \in [\zero, \ba], M_\bb \ne 0 \, \}. 
\end{align*}

Let $M,N \in \mod_{\ZZ^n} S$. If there is a $\ZZ^n$-graded $\kk$-linear bijection $f:M_{[\zero, \ba]} \to N_{[\zero, \ba]}$ 
satisfying $f(x^{\bd-\be} y)=x^{\bd-\be} \cdot f(y)$ for all $\bd, \be \in [\zero, \ba]$ with $\bd \succeq \be$ and all $y \in M_\be$, 
we say $M_{[\zero, \ba]}$ and $N_{[\zero, \ba]}$ are isomorphic (over $S$).  
If $M,N \in \mod_\ba S$ and $M_{[\zero, \ba]} \cong N_{[\zero, \ba]}$, we have $M \cong N$.

Recall that, for $Z \subset \{ x_1, \ldots, x_n \}$, $\kk[Z]$ denotes 
the $\kk$-subalgebra of $S$ generated by all $x_i \in Z$. 
To make $\kk[Z]$ an $S$-module, set $x_i \cdot \kk[Z]=0$ for all $x_i \not \in Z$. 
In other words, $\kk[Z] \cong S/(x_i \mid x_i \not \in Z)$. 
When we regard a Stanley decomposition $\cD= \bigoplus_{i=1}^s m_i  \, \kk[Z_i]$ of $M$ 
as an $S$-module, it is denoted by $|\cD|$.   
We say $\cD$ is {\it positively $\ba$-determined}, if the module $|\cD|$ is positively $\ba$-determined, 
equivalently, $\zero \preceq \deg(m_i) \preceq \ba$ and $\supp^\ba_X(m_i) \subset Z_i$ 
for all $1 \leq i \leq s$. 
If $M$ admits such a decomposition, then $M$ itself is positively $\ba$-determined. 
For $M \in \mod_\ba S$, let $\sd_\ba(M)$ be the set of positively $\ba$-determined Stanley 
decompositions of $M$. If $M \in \mod_\ba S$, then 
$\sd_{\ba'} (M) \supset \sd_\ba (M)$ for $\ba' \in \NN^n$ with $\ba' \succeq \ba$, and  
$$\sd(M)=\bigcup_{\ba \in \NN^n} \sd_\ba (M).$$

\begin{prop}\label{sda}
For $M \in \mod_\ba S$, we have
$$\sdepth M = \max \{ \, \sdepth \cD \mid \cD \in \sd_\ba(M) \, \}.$$
\end{prop}

If $M$ is a squarefree module (i.e., if $\ba =\one$), the above result has been proved by
Soleyman Jahan (\cite[Theorem~3.4]{SoJ})

\begin{proof}
Since $\sd_\ba(M) \subset \sd(M)$, the inequality  
$\sdepth M \geq \max \{ \, \sdepth \cD \mid \cD \in \sd_\ba (M)\}$ is clear.  
To prove the converse inequality, from $\cD= \bigoplus_{i=1}^s m_i \, \kk[Z_i] \in \sd (M)$, 
we will construct $\cD' \in \sd_{\ba}(M)$ with 
$\sdepth \cD' \geq \sdepth \cD$.  
We may assume that $\deg(m_i) \preceq \ba$ for all $1 \leq i \leq t$, and 
$\deg(m_i) \not \preceq \ba$ for all $i > t$.
Set $$\cD' := \bigoplus_{i=1}^t m_i \,\kk[Z_i \cup \supp^\ba_X(m_i)].$$ 
Then  $m_i \,\kk[Z_i \cup \supp^\ba_X(m_i)]$ is a Stanley space for each $i$. 
Since $|\cD'|_{[\zero, \ba]} \cong |\cD|_{[\zero, \ba]}$ and $M \in \mod_\ba S$, 
we have $\cD' \in \sd_\ba (M)$. 
It is clear that $\sdepth \cD' \geq \sdepth \cD$. 
\end{proof}

For $M \in \mod_{\ZZ^n} S$ and $\bb \in\ZZ^n$, 
let $\gb_{i,\bb}(M):=\dim_\kk (\opn{Tor}_i^S(\kk, M))_\bb$ be the $(i,\bb)^{\text{th}}$ 
graded betti number of $M$. 

\begin{dfn}[\cite{M}]
For $M \in \mod_{\NN^n} S$, the {\it support regularity} of $M$ is 
$$
\sreg(M) := \max\set{\# \supp(\bb) - i}{\gb_{i,\bb}(M) \neq 0}. 
$$
\end{dfn}

\begin{rem}\label{suupp reg basic}
The inequalities in \cite[Corollary~20.19]{E}, which is a basic property of 
the usual (Castelnuovo-Mumford) regularity 
$$
\reg_S(M) := \max\set{j - i}{\gb_{i,j}(M) \neq 0} 
$$
of a finitely generated $\ZZ$-graded $S$-module $M$,  
also holds for the support regularity. In the proof in \cite{E}, the long exact sequence of 
$\Ext_S^i(-,S)$ is used to handle the regularities,  but we can use that of 
$\opn{Tor}_i^S(-, \kk)$. Then the same argument works for the support regularity. 
\end{rem}

Miller (\cite{M}) introduced the {\it Alexander duality functor} $\sA_\ba : \mod_\ba S \to 
(\mod_\ba S)^\op$, which is an exact functor with $(\sA_\ba)^2=\opn{Id}$.  
For  $M \in \mod_\ba S$, $\bb \in [\zero, \ba]$ and $i \in \supp (\bb)$, 
we have $(\sA_{\ba}(M))_\bb = \Hom_\kk (M_{\ba-\bb}, \kk)$ 
and the multiplication map $(\sA_\ba(M))_{\bb - \be_i} \ni y \longmapsto x_i y \in (\sA_\ba(M))_\bb$ is the 
$\kk$-dual of $M_{\ba-\bb} \ni z \longmapsto x_i z \in M_{\ba-\bb+\be_i}$. 
We have that 
$$
\dim_S (\sA_\ba(M)) + \gs(M) =n,
$$ 
where 
$$
\gs(M):= \min \{ \, \# \supp(\bb) \mid M_\bb \ne 0 \, \}. 
$$
See \cite{M} for further information. 
In the sequel, we sometimes omit the suffix $\ba$ of $\sA_\ba$,  
if the explicit value of $\ba$ is not important.

\begin{thm}[{\cite[Theorem 4.20]{M}}]\label{sreg}
For $M \in \mod_\ba S$, we have
$$
\sreg(M)+\depth (\sA_\ba(M))=n.
$$
\end{thm}

Note that $\sreg (M) \geq \gs(M)$ for all $M \in \mod_\ba S$. 
By Theorem~\ref{sreg}, $\sreg (M) = \gs(M)$ if and only if 
$\sA_\ba(M)$ is Cohen-Macaulay.

\section{Alexander duality and (quasi) Stanley decomposition}
For $\ba, \bb, \bc \in \NN^n$ with $\bc \preceq \bb \preceq \ba$, 
we set
\begin{eqnarray*}
\kk_{\ba}[\bc, \bb] &:=& x^\bc \cdot (S/(x_i^{b_i + 1} \mid i \not \in \supp^\ba(\bb)))\\
 &\cong &(S/(x_i^{b_i -c_i+ 1} \mid i \not \in \supp^\ba(\bb)))(-\bc).
\end{eqnarray*}
This is an ideal of $S/(x_i^{b_i + 1} \mid i \not \in \supp^\ba(\bb))$. 
Set $$[\![\bc, \bb ]\!]_\ba := \{ \, \bd \in \NN^n \mid  (\kk_{\ba}[\bc, \bb])_\bd \ne 0 \, \}.$$
We see that $\bd \in [\![\bc, \bb ]\!]_\ba$ if and only if $\bd \succeq \bc$ and 
$d_i \leq b_i$ for all $i \not \in \supp^\ba(\bb)$. 
For $\bd \in  [\![\bc, \bb ]\!]_\ba$, the natural image of the monomial 
$x^\bd \in S$  in $\kk_\ba[\bc, \bb] \subset S/(x_i^{b_i + 1} \mid i \not \in \supp^\ba(\bb))$ is denoted by 
by $\bax^\bd$. (This is an abuse of notation, since the symbol $\bax^\bd$ ignores $\bb$ and $\bc$.) 
It is easy to check that $\kk_\ba[\bc,\bb] \in \mod_\ba S$ with 
\begin{equation}\label{k[c,b] support}
(\kk_{\ba}[\bc, \bb])_{[\zero, \ba]} = \bigoplus_{\bd \in [\bc, \bb]} \kk \, \bax^\bd.  
\end{equation}

\begin{lem}\label{k[c,b] alex dual}
We have $\sA_\ba(\kk_\ba[\bc, \bb])\cong \kk_\ba[\ba-\bb, \ba-\bc]$.
\end{lem}

\begin{proof}
By \eqref{k[c,b] support}, we have 
$$(\sA_\ba(\kk_\ba[\bc, \bb]))_{[\zero, \ba]} =  \bigoplus_{\bd \in [\ba-\bb, \ba-\bc]} \kk \, t_\bd$$
as a $\ZZ^n$-graded $\kk$-vector space, where $t_\bd$ is the dual base of 
$\bax^{\ba-\bd} \in (\kk_\ba[\bc, \bb])_{\ba-\bd}$ and has the degree  $\deg(t_\bd)=\bd$. 
For $\bd, \be \in [\bc, \bb]$ with $\bd \succeq \be$, we have $x^{\bd-\be} \cdot \bax^\be=\bax^\bd$ 
in $\kk_\ba[\bc, \bb]$. Hence we have $x^{\bd-\be} \cdot t_{\ba-\bd}=t_{\ba-\be}$ in 
$\sA_\ba(\kk_\ba[\bc, \bb])$.  It follows that $(\sA_\ba(\kk_\ba[\bc, \bb]))_{[\zero, \ba]}\cong 
(\kk_\ba[\ba-\bb, \ba-\bc])_{[\zero, \ba]}$. Since both $\sA_\ba(\kk_\ba[\bc, \bb])$ and  
$\kk_\ba[\ba-\bb, \ba-\bc]$ are positively $\ba$-determined, we have 
$\sA_\ba(\kk_\ba[\bc, \bb])\cong \kk_\ba[\ba-\bb, \ba-\bc]$.
\end{proof}

\begin{dfn}
Let $M \in \mod_{\NN^n} S$. We say $f: \bigoplus_{i=1}^s \kk_\ba[\bc_i, \bb_i] \to M$ is a 
{\it (positively $\ba$-determined) quasi Stanley decomposition}, 
if $f$ is a $\ZZ^n$-graded bijective $\kk$-linear map 
such that $f(\bax^\bd)= x^{\bd-\bc_i} \cdot f(\bax^{\bc_i})$ for all $i$ and 
all $\bax^\bd \in \kk_\ba[\bc_i, \bb_i]$ with $\bd \in  [\![\bc_i, \bb_i ]\!]_\ba$.  
\end{dfn}

Let $\qsd_\ba (M)$ be the set of positively $\ba$-determined quasi Stanley 
decompositions of $M$. For a decomposition 
$f: \cD \to M$, $\cD =\bigoplus_{i=1}^s \kk_\ba[\bc_i, \bb_i]$, we write  
$(\cD, f) \in \qsd_\ba (M)$ or just $\cD \in \qsd_\ba (M)$. 
If $\qsd_\ba (M) \ne \emptyset$, then $M \in \mod_\ba S$. 
Conversely, if $M \in \mod_\ba S$, then we can replace the condition 
$\bd \in  [\![\bc_i, \bb_i ]\!]_\ba$ by $\bd \in  [\bc_i, \bb_i]$ 
in the above definition. 
Let $f_i$ be the restriction of the map $f:\cD \to M$ to $\kk_\ba[\bc_i, \bb_i]$. 
Note that $f_i$ is just a $\poly{\supp^\ba_X(\bb_i)}$-homomorphism, and
{\it not} a $\kk[\supp_X (\bb_i)]$-homomorphism. See Example~\ref{qsd exmp} below.

For $M \in \mod_\ba S$, $\sd_\ba(M)$ can be seen as a subset of $\qsd_\ba(M)$ in the natural way. 
In fact, for $\bigoplus_{i=1}^sm_i \, \kk[Z_i] \in \sd_\ba(M)$, 
set $\bc_i := \deg(m_i) \in \NN^n$ (since the decomposition is positively $\ba$-determined, 
we have $(c_i)_j < a_j$ for all $j \not \in Z_i$), and 
take $\bb_i \in \NN^n$ whose $j^{\rm th}$ coordinate is 
\begin{equation}\label{abc}
(b_i)_j =\begin{cases}
a_j & \text{if $j \in Z_i$,} \\
(c_i)_j & \text{otherwise.}
\end{cases}
\end{equation}
Finally, define $f: \bigoplus_{i=1}^s \kk_\ba[\bc_i, \bb_i] \to M$ by 
$\kk_\ba[\bc_i, \bb_i] \ni \bax^\bd  \longmapsto x^{\bd -\bc_i}  \cdot m_i \in M$ 
for $\bd \in  [\![\bc_i, \bb_i ]\!]_\ba$. 
Then we have $(\bigoplus_{i=1}^s \kk_\ba[\bc_i, \bb_i], f) \in \qsd_\ba(M)$.  

In the sequel, for $\bb,\bc \in [\zero, \ba]$ satisfying the same condition as \eqref{abc},   
$\kk_\ba[\bc, \bb]$ is denoted by $x^{\bc} \, \kk[ \supp^\ba_X(\bb)]$. 

\begin{exmp}\label{qsd exmp}
Let $I:= (x^3, x^2y)$ be a monomial ideal of $S:= \kk[x,y]$, and 
set $\ba:= (3,1)$. Then $S/I \in \mod_\ba S$ and $\{ \, y^l, xy^m, x^2 \mid l, m \in \NN \, \}$ 
is a $\kk$-basis of $S/I$. It is easy to check that 
$$\kk_\ba [\, \zero, (1,1) \, ] \oplus \kk_\ba[\, (2,0),(2,0) \,]$$
is a quasi Stanley decomposition of $S/I$, but not a Stanley decomposition. 
Note that $\kk_\ba[\, \zero, (1,1) \, ] \cong S/(x^2)$ 
and $\kk_\ba[\, (2,0),(2,0) \,] \cong \kk(-(2,0))$. 
While $\supp_X((1,1))= \{x,y\}$, the corresponding map  $S/(x^2) \to S/I$ 
is not an $S$-homomorphism (just a $\kk[y]$-homomorphism).  
\end{exmp}

\begin{lem}\label{k[b,c] basic} 
Let $\ba, \bb, \bc \in \NN^n$ with $\bc \preceq \bb \preceq \ba$. 
Then $$\sdepth (\kk_\ba[\bc,\bb])= \# \supp^\ba(\bb).$$   
\end{lem}

\begin{proof}
Since $\kk_\ba[\bc,\bb] \cong S/(x_i^{b_i -c_i+ 1}\ |\ i \not \in  \supp^\ba(\bb))$ up to degree shifting, 
the assertion follows from \cite[Theorem~3]{A2}.  However, we will give a direct proof here 
for the reader's convenience. 

Since $\dim_S  (\kk_\ba[\bc,\bb]) =\# \supp^\ba(\bb)$,  
it suffices to show that $\sdepth (\kk_\ba[\bc,\bb]) \geq \# \supp^\ba(\bb)$. 
This inequality follows from the Stanley decomposition 
$$\kk_\ba[\bc, \bb] = \Dsum x^{\bc'}\poly{\supp_X^{\ba}(\bb)},$$
where the sums are taken over $\bc' \in \bra{\bc,\bb}$ such that
$c'_i = c_i$ if $i \in \supp^\ba (\bb)$ and $c_i \le c'_i \le b_i$ otherwise.
\end{proof}

\begin{dfn}
For a quasi Stanley decomposition $\cD = \bigoplus_{i=1}^s \kk_\ba[\bc_i, \bb_i]$ 
of $M \in \mod_\ba S$, we set 
$$\sdepth \cD = \min\{ \# \supp^\ba(\bb_i) \mid 1 \leq i \leq s  \}.$$
(If $\cD \in \qsd (M)$ comes from a Stanley decomposition, this definition clearly 
coincides with the previous one.)
\end{dfn}

\begin{rem}
In the above definition, $\sdepth \cD$ is the Stanley depth of $\cD$ 
{\it as a decomposition}. By Lemma~\ref{k[b,c] basic}, we have  $\sdepth |\cD|  \ge \sdepth \cD$.  
The authors do not know whether the equality always holds or not.
\end{rem}

\begin{prop}\label{sdepth by qsd}
For $M \in \mod_\ba S$, we have 
$$\sdepth M = \max \{ \, \sdepth \cD \mid \cD \in \qsd_\ba (M) \, \}.$$
\end{prop}

\begin{proof}
Since $\sd_\ba (M) \subset \qsd_\ba(M)$, we have  
$\sdepth M \leq \max \{ \, \sdepth \cD \mid \cD \in \qsd_\ba (M) \, \}$ by Proposition~\ref{sda}. 
To show the converse inequality, take a decomposition $(\cD, f) \in \qsd_\ba(M)$ with 
$\cD= \bigoplus_{i=1}^s \kk_\ba[\bc_i, \bb_i]$. As Lemma~\ref{k[b,c] basic}, take a Stanley decomposition 
$\cD_i$ of $\kk_\ba[\bc_i, \bb_i]$ for each $1 \le i \le s$. Since the restriction of $f :\cD \to M$ 
to $\kk_\ba[\bc_i, \bb_i]$ is a $\kk[ \supp^\ba_X(\bb_i)]$-homomorphism, $f(\cD_i)$ is a direct sum of 
Stanley spaces. On the other hand, $\bigoplus_{i=1}^s f(\cD_i) = \bigoplus_{i=1}^s f(\kk_\ba[\bc_i, \bb_i])=M$ 
as $\ZZ^n$-graded $\kk$-vector spaces.       
Hence $\cD':= \bigoplus_{i=1}^s f(\cD_i)$ is a Stanley decomposition of $M$, and we have 
$\sdepth M \ge \sdepth \cD' =\sdepth \cD$.   
\end{proof}

From a decomposition $(\cD, f) \in \qsd_\ba (M)$ with  
$\cD =\bigoplus_{i=1}^s \kk_\ba[\bc_i, \bb_i]$ 
of $M \in \mod_\ba S$, we will construct its Alexander dual 
$(\sA_\ba(\cD), g) \in \qsd_\ba (\sA_\ba(M))$ with 
$\sA_\ba (\cD) =\bigoplus_{i=1}^s \kk_\ba[\ba-\bb_i, \ba-\bc_i]$. 
Note that $|\sA_\ba (\cD)| \cong \sA_\ba (|\cD|)$ by Lemma~\ref{k[c,b] alex dual} 
and $(\sA_\ba(\cD))_{\ba-\bd}= \Hom_\kk(\cD_\bd, \kk) 
=: (\cD_\bd)^*$ for each $\bd \in [\zero, \ba]$. For this $\bd$, 
set $T(\bd):= \{ \, i  \mid \bc_i \preceq  \bd \preceq \bb_i \, \} \subset 
\{ 1, \ldots, n \}$. Then $\cD_\bd$ and $M_\bd$ have the basis 
$\{ \, \bax^\bd \in \kk[\bc_i, \bb_i] \mid i \in T(\bd)\, \}$
and $\{ \, f(\bax^\bd)  \mid i \in T(\bd), \, \bax^\bd \in \kk[\bc_i, \bb_i] \, \}$ respectively. 
Of course, the equations $\bb_i=\bb_j$ and $\bc_i=\bc_j$ might hold for distinct $i,j$. 
Even in this case, we distinguish $\bax^{\bd} \in \kk_\ba[\bc_i, \bb_i]$ from 
$\bax^{\bd} \in \kk_\ba[\bc_j, \bb_j]$. For the convenience, 
$\bax^\bd_i$ denotes $\bax^{\bd} \in \kk_\ba[\bc_i, \bb_i]$. 

Note that  $(\sA_\ba(M))_{\ba-\bd}$ has the dual basis 
$\{ \, f(\bax^\bd_i)^*  \mid i \in T(\bd) \, \}$. 
Now we can define a $\kk$-linear bijection 
$$g_{\ba-\bd} : \left( \bigoplus_{i=1}^s \kk_\ba[\ba-\bb_i, \ba-\bc_i]\right)_{\ba-\bd} 
\longrightarrow (\sA_\ba(M))_{\ba-\bd}$$ by 
$$(\kk_\ba[\ba-\bb_i, \ba-\bc_i])_{\ba-\bd} \ni \bax^{\ba-\bd} \longmapsto 
f(\bax^\bd_i)^*  \in (\sA_\ba(M))_{\ba-\bd}$$
for $i \in T(\bd)$ (note that $(\kk_\ba[\ba-\bb_i, \ba-\bc_i])_{\ba-\bd} \ne 0$ if and only if 
$(\kk_\ba[\bc_i, \bb_i])_{\bd} \ne 0$ if and only if $i \in T(\bd)$).  
It is easy to see that $g:= \bigoplus_{\bd \in [\zero, \ba]} g_\bd$ 
gives a $\kk$-linear bijection $(\sA_\ba(\cD))_{[\zero, \ba]} \to (\sA_\ba(M))_{[\zero, \ba]}$ 
satisfying $x^{\bd-\be}\cdot g(\bax^{\ba-\bd})=  g(\bax^{\ba-\be})$ for all 
 $\bd, \be \in [\bc_i, \bb_i]$ with $\bd \succeq \be$. Here $\bax^{\ba-\bd}, \bax^{\ba-\be} \in 
\kk_\ba[\ba-\bb_i, \ba-\bc_i]$. 
Since both $\sA_\ba(M)$ and $|\sA_\ba(\cD)|$ are positively $\ba$-determined modules, 
we can extend $g$ to a $\kk$-linear bijection $\sA_\ba(\cD) \to \sA_\ba(M)$ so that 
$\sA_\ba(\cD) \in \qsd_\ba(\sA_\ba(M))$. 
Now we have the following. 
 
\begin{prop}
The above construction 
gives a one-to-one correspondence between $\qsd_\ba(M)$ and $\qsd_\ba(\sA_\ba(M))$.  
\end{prop}

\begin{rem}
If $M$ is  squarefree (i.e., $M \in \mod_\one S$), then $\qsd_\one (M)= \sd_\one (M)$ 
and the Alexander duality $\sA_\one$ gives a duality between $\sd_\one(M)$ and $\sd_\one(\sA_\one(M))$. 
This is the reason why the notion of quasi Stanley decompositions does not appear in \cite{SoJ}, 
while the Alexander duality of Stanley decompositions is studied there.  
\end{rem}

For $\ba,\ba',\bb,\bc \in \NN^n$ with $\bc \preceq \bb \preceq \ba \preceq \ba'$, 
we have $$\kk_{\ba}[\bc,\bb] = \kk_{\ba'}[\bc, \bb'],$$ 
where $\bb' \in \NN^n$ is the vector whose $i^{\rm th}$ coordinate is 
$$
b'_i=\begin{cases}
a'_i & \text{if $b_i = a_i$,} \\
b_i & \text{otherwise (equivalently, $b_i < a_i$).}
\end{cases}
$$
If $M \in \mod_\ba S$ and $\ba' \succeq \ba$, then $M \in \mod_{\ba'} S$ and 
$\qsd_\ba (M)$ can be seen as a subset of $\qsd_{\ba'} (M)$ in the natural way. 
Set 
$$\qsd(M):=\bigcup_{\ba \in \NN^n} \qsd_\ba (M).$$

As the Stanley depth is (conjectured to be) a combinatorial analog of 
the usual depth, the invariant $\shreg (M)$ defined below is a combinatorial analog of 
$\sreg (M)$. Note that $\sreg (\kk_\ba[\bc, \bb])=\# \supp(\bc)$. In fact, $\kk_\ba[\bc, \bb] 
\cong (S/(x^{b_i-c_i+1} \mid i \not \in \supp^\ba(\bb)))(-\bc)$, 
and the Koszul complex (with the degree shift) gives a minimal free resolution.

\begin{dfn}\label{shreg dfn}
For $\cD = \bigoplus_{i=1}^s \kk_\ba[\bc_i, \bb_i]$, set
$$\shreg (\cD) := \max \{ \, \#  \supp (\bc_i) \mid 1 \leq i \leq s \, \}.$$
For $M \in \mod_{\NN^n} S$, set 
$$\shreg (M):=  \min \{ \, \shreg (\cD) \mid \cD \in \qsd(M) \, \}.$$ 
\end{dfn}

\begin{lem}
If $M \in \mod_\ba S$, we have 
$$\shreg M =\min \{ \, \shreg (\cD) \mid \cD \in \qsd_\ba (M)\}.$$ 
\end{lem}

\begin{proof}
Since $\qsd_\ba(M) \subset \qsd(M)$, we see that 
$\shreg M \leq \min \{ \, \shreg \cD \mid \cD \in \qsd_\ba (M)\}$. 
To prove the converse inequality,  
from $(\cD', f') \in \qsd_{\ba'}(M)$, 
we will construct $(\cD, f) \in \qsd_{\ba}(M)$ with 
$\shreg \cD \leq \shreg \cD'$.  
Replacing $\ba'$ by $\ba \vee \ba'$ if necessary, we may assume that $\ba' \succeq \ba$ 
(note that $\qsd_{\ba'} (M) \subset \qsd_{\ba \vee \ba'}(M)$). 
Set $\cD' = \Dsum_{i=1}^s  \kk_{\ba'}[\bc_i,\bb_i]$. 
We may assume that $\bc_i \preceq \ba$ for all $1 \leq i \leq t$ 
and  $\bc_i \not \preceq \ba$ for all $i > t$.  
Set $\cD:= \Dsum_{i=1}^t \kk_{\ba}[\bc_i, \bb_i \wedge \ba]$.  
Since $\cD_{[\zero, \ba]} \cong \cD'_{[\zero, \ba]}$ and  
$M \in \mod_\ba S$, we can define  $f: \cD \to M$ by 
$\kk_\ba [\bc_i, \bb_i \wedge \ba] \ni \bax^\bd \longmapsto x^{\bd-\bc_i} \cdot f'(\bax^{\bc_i}) \in M$ for all $\bd \in [\![ \bc_i, \bb_i \wedge \ba]\!]_\ba$.  
Then $(\cD, f)$ has the expected properties.  
\end{proof}

\begin{rem}\label{shreg remark}
(1) To compute $\shreg M$, the notion of quasi Stanley decompositions is really necessary.   
For example, set $S:= \kk[x,y]$, $\ba:= (1,2)$, and $M:=\kk_\ba[\zero,(0,1)] \cong S/(x, y^2)$. 
Then $M$ has a trivial quasi Stanley decomposition, and $\shreg M =0$. 
However $\cD=\kk \oplus y \, \kk$ is the unique Stanley decomposition of $M$, 
and $\shreg \cD =1$. 

(2) For a Stanley decomposition $\cD = \bigoplus_{i=1}^s m_i \, \kk[Z_i] \in \sd(M)$ with $\deg(m_i)=\bc_i$, 
Soleyman Jahan (\cite{SoJ}) set $\hreg (\cD) := \max \{ \, |\bc_i| \mid 1 \leq i \leq s \, \},$
where $|\bc_i|:= \sum_{j=1}^n (c_i)_j$ is the total degree of $\bc_i$. 
He also set $\hreg M:=  \min \{ \, \hreg \cD \mid \cD \in \sd(M) \, \}.$ 
Clearly, we have $\shreg M \leq \hreg M$ and the inequality is strict quite often. 
However, if $M$ is squarefree, then $\shreg M = \hreg M$. 
For squarefree modules, \cite[Conjecture~4.3]{SoJ} is equivalent to the condition (iii) 
of Theorem~\ref{skeleton CM} below. 
\end{rem}

\begin{thm}\label{shreg}
If $M \in \mod_{\ba} S$, then we have 
$$\shreg(M) + \sdepth(\sA_\ba(M))=n.$$
\end{thm}

\begin{proof}
For $\cD = \bigoplus_{i=1}^s \kk_\ba[\bc_i, \bb_i] \in \qsd_\ba(M)$, we have 
\begin{eqnarray*}
n - (\shreg \cD) &=& n - \max \{ \, \# \supp (\bc_i) \mid  1 \leq i \leq s \, \}\\
&=& \min \{ \, n- \# \supp (\bc_i) \mid  1 \leq i \leq s \, \}\\
&=& \min \{ \, \# \supp^\ba (\ba-\bc_i) \mid 1 \leq i \leq s  \}\\
&=& \sdepth(\sA_\ba(\cD)). 
\end{eqnarray*} 

Hence we have 
\begin{align*}
n - (\shreg M) &= n - \min \{ \, \shreg \cD \mid \cD \in \qsd_\ba(M) \, \} \\
               &= \max\{ \, n- (\shreg \cD) \mid \cD \in \qsd_\ba(M) \, \} \\
               &= \max \set{\sdepth(\sA_\ba(\cD))}{\cD \in \qsd_\ba(M)} \\
               &= \max\set{\sdepth (\cD')}{\cD' \in \qsd_\ba(\sA_\ba(M))} \\
               &= \sdepth (\sA_\ba(M)).              
\end{align*}
\end{proof}

\begin{cor}
For a short exact sequence 
$0 \longto L \longto M \longto N \longto 0$ in $\mod_{\NN^n} S$, we have 
$\shreg M \le \max \{\, \shreg L, \, \shreg N \, \}.$
\end{cor}

\begin{proof}
Since we have the exact sequence $0 \longto \sA(N) \longto \sA(M) \longto \sA(L) \longto 0$, 
the assertion follows from Lemma~\ref{sec:ex seq and sdepth} and Theorem~\ref{shreg}. 
\end{proof}

\section{Skeletons of positively $\ba$-determined modules}\label{sec:skel}
Let $M \in \mod_\ba S$. For $l \ge 0$, let $M^{>l}$ be the submodule of $M$ generated 
by the components $M_\bb$ for all $\bb \in \NN^n$ with $\# \supp^\ba(\bb) > l$. 
The module $M^{>l}$ is again positively $\ba$-determined. We set
$$
M^{\le l} := M/M^{>l},
$$
and call it the {\em $l^{\text{th}}$ skeleton of $M$}.
Clearly, $M^{\le l}$ is a positively $\ba$-determined module
with $\dim_S M^{\le l} \le l$, and $M^{\le l} = M$ for $l \ge \dim_S M$. 

\begin{rem}
(1) For a simplicial complex $\gD$ with the vertex set $\{1, \ldots, n\}$, the {\it Stanley-Reisner ring} 
$\kk[\Delta]$ of $\Delta$ is defined to be $S/(\prod_{i \in F} x_i \mid F \not \in \Delta)$. 
Then $\dim\kk[\Delta] = \max \{ \# F \mid F \in \Delta\}= \dim \Delta +1$. 
Moreover, $\kk[\Delta]$ is always a squarefree module, that is, $\kk[\Delta] \in \mod_\one S$. 
In this setting, we have $\poly{\gD}^{\le l} = \poly{\gD^{(l-1)}}$, where 
$\Delta^{(l-1)} := \{ F \in \Delta \mid \#F \leq l\}$ is the $(l-1)^{\rm st}$ skeleton of $\Delta$. 

(2) Let $I$ be a monomial ideal minimally generated by $x^{\ba_1},\ldots, x^{\ba_r}$. 
In the sequel, the skeleton of a module means the one with respect to $\ba = \ba_1 \vee \cdots \vee \ba_r$. 
Then $J:= I+S^{>l}$ coincide with the $l^{\text{th}}$ {\em skeleton ideal} of $I$ 
due to Herzog et al. (\cite{HSZ}). Note that $S/J \cong (S/I)^{\leq l}$. 
\end{rem}

\begin{lem}\label{sec:lem_of_skel}
Let $M \in \mod_\ba S$ and $l \ge 0$. If $M^{> l-1} \neq M^{> l}$, then 
$M^{> l -1}/M^{> l}$ is a Cohen-Macaulay module of dimension $l$. 
Moreover, $\sdepth(M^{>l-1}/M^{>l}) = l$.
\end{lem}
\begin{proof}
We set $\tilde M := M^{>l-1}/M^{>l}$. For $\bb \in \NN^n$,  $\tilde{M}_\bb \ne 0$ 
implies $\#\supp^\ba(\bb) = l$ and $\tilde{M}_\bb =M_\bb$. 
For $F \subseteq \bra n := \mbra{1,\dots,n}$ with $\# F =l$, set 
$$\tilde M_{[F]}:=  \Dsum_{\substack{\bb \in \NN^n \\ \supp^\ba(\bb) = F}} M_\bb.$$
Then $\tilde M_{[F]}$ is an $S$-submodule of $\tilde{M}$, and  we have  
$$
\tilde M = \Dsum_{\substack{F \subseteq \bra n \\ \# F = l}} \tilde M_{[F]},
$$
as $S$-modules. 
If we regard $\tilde M_{[F]}$ as an $S':=\poly{x_i\mid i \in F}$-module through 
the natural injection $S' \hookrightarrow S$, then  
$\tilde M_{[F]}$ is a finite free $S'$-module 
with 
$$
\tilde M_{[F]} \cong \Dsum_{\substack{\bb \in \bra{\zero, \ba} \\ \supp^\ba(\bb) = F}}
                    (S'(-\bb))^{\dim_\kk(M_\bb)}. 
$$
Therefore $\tilde{M}$ is a Cohen-Macaulay module of dimension $l$ over $S'$,  
hence the same is true over $S$. 
The above decomposition also shows that $\sdepth \tilde{M} = l$.  
\end{proof}

As in the case of the skeletons of monomial ideals, the following holds.

\begin{prop}[cf. {\cite[Corollary 1.5]{HSZ}}]\label{sec:depth_and_skel}
For $0 \neq M \in \mod_\ba S$,
$$
\depth M = \max\set{l}{0 \le l \le \dim_S M,\ M^{\le l} \text{ is Cohen-Macaulay}}.
$$
Moreover, we have $\dim_S M^{\leq \depth M} = \depth M$.  
\end{prop}

\begin{proof}
We use induction on $d:= \dim_S M$. The case $d = 0$ is trivial. Assume $d > 0$.
The assertion clearly holds when $M$ is Cohen-Macaulay.
Hence it suffices to consider the case $\depth M < d$. 
Since $M^{>d}=0$, $M^{>d-1} (= M^{> d-1}/M^{>d})$ is a Cohen-Macaulay module of dimension $d$ 
by Lemma~\ref{sec:lem_of_skel}.  By the short exact sequence
$$
0 \longto M^{> d- 1} \longto M \longto M^{\le d- 1} \longto 0,
$$
we have $\depth M = \depth M^{\le d- 1}$. 
On the other hand, we have $M^{\le l} \cong (M^{\le d-1})^{\le l}$
for all $l \leq d-1$. Combining the above facts, we have 
\begin{align*}
\depth M &= \depth M^{\le d-1} \\
&= \max\set{l}{0 \le l \le d-1, \ (M^{\le d-1})^{\le l} \text{ is Cohen-Macaulay}} \\ 
&= \max\set{l}{0 \le l \le d-1, \ M^{\le l} \text{ is Cohen-Macaulay}}\\
&= \max\set{l}{0 \le l \le d, \ M^{\le l} \text{ is Cohen-Macaulay}}. 
\end{align*}
Here, the second equality follows from the induction hypothesis, 
and the fourth follows from the present assumption that $M^{\le d} (=M)$ is not Cohen-Macaulay. 

That $\dim_S M^{\leq \depth M} = \depth M$ also follows from similar argument.  
\end{proof}

We can also prove that $M^{\leq l}$ is Cohen-Macaulay (or the 0 module) 
for all $l \leq \depth M$, while we do not use this fact in this paper.

\begin{lem}\label{sec:lem_sdep}
For $\bb, \bc \in \bra{\zero,\ba}$ with $\bc \preceq \bb$,
we have $\sdepth( \kk_\ba[\bc,\bb]^{\le l}) = l$ if 
$\# \supp^\ba(\bc) \le l \le \# \supp^\ba(\bb)$. 
\end{lem}

\begin{proof}
We use induction on $l$ starting from  $l= \# \supp^\ba (\bb)$. 
If $l= \# \supp^\ba (\bb)$, then $\kk_\ba[\bc,\bb]^{\le l} = \kk_\ba[\bc,\bb]$, and the assertion 
has been shown in Lemma~\ref{k[b,c] basic}.
Consider the case $l < \# \supp^\ba (\bb)$. 
Since $\sdepth( \kk_\ba[\bc,\bb]^{\le l}) \leq \dim_S ( \kk_\ba[\bc,\bb]^{\le l}) =l$, 
it suffices to show  $\sdepth( \kk_\ba[\bc,\bb]^{\le l}) \geq l$. 
We have $\sdepth(\kk_\ba[\bc, \bb]^{\le l+1})=l+1$ by the induction hypothesis,
and there exists a decomposition $\cD:=\bigoplus_{i=1}^s x^{\bc_i} \, \kk[Z_i] \in 
\sd_\ba(\kk_\ba[\bc,\bb]^{\le l+1})$ with $\# Z_i = l+1$ for all $i$.  
Since $\cD$ is positively $\ba$-determined, we have $\supp^\ba(\bc_i) \subset Z_i$ for all $i$.  
Note that $\kk_\ba[\bc, \bb]^{\le l} = (\kk_\ba[\bc, \bb]^{\le l+1})^{\le l} 
= \bigoplus_{i=1}^s (x^{\bc_i} \, \kk[Z_i])^{\le l}$ as $\ZZ^n$-graded $\kk$-vector spaces. 
Hence, if $\cD_i$ is a Stanley decomposition of $(x^{\bc_i} \, \kk[Z_i])^{\le l}$, 
then $\bigoplus_{i=1}^s \cD_i$ is a Stanley decomposition of $\kk_\ba[\bc, \bb]^{\le l}$ 
by an argument similar to the proof of Proposition~\ref{sdepth by qsd}. 
Therefore the problem can be reduced to the case 
$\kk_\ba[\bc,\bb]= x^\bc \, \kk[Z]$ with $\# Z =l+1$ and $\supp^\ba(\bc) \subset Z$. 
If $\supp^\ba(\bc)= Z$, then $\kk_\ba[\bc,\bb]^{\le l}=0$ and there is nothing to prove. 
So we may assume that $\supp^\ba(\bc) \subsetneq Z$. 
Define $\bb' \in \ZZ^n$ as follows;
$$
b_i' := \begin{cases}
a_i - c_i  &\text{if $i \in Z $;} \\
0 &\text{otherwise.}
\end{cases}
$$
It is easy to verify that
$$
\left(\poly Z/x^{\bb'} \, \poly Z \right) (-\bc) \cong \left(x^{\bc} \, \poly Z\right)^{\le l}. 
$$
Since $\poly Z/x^{\bb'} \poly Z$ can be seen as the quotient ring of $S$ by the complete intersection ideal 
$I = (x^{\bb'})+ ( \, x_i \mid x_i \not \in Z \, )$, 
Stanley's conjecture  holds for $\poly Z/x^{\bb'} \poly Z$ ($\cong S/I$) by \cite[Theorem~3]{A2}.
(We can prove this statement directly using the results in the next section.
In fact, we can reduce to the case $\bb' \preceq \one$.) 
Thus we have 
$$\sdepth \left(x^{\bc}\poly Z\right)^{\le l} = 
\sdepth \left(\poly Z/x^{\bb'} \poly Z \right) =  l,$$
as desired.
\end{proof}

Now we have the following.

\begin{prop}\label{depth skeleton}
For $M \in \mod_{\ba} S$, $\sdepth M \ge t$ if and only if $\sdepth M^{\le t} \ge t$.
\end{prop}
\begin{proof}
To see the ``only if" part, take  $\cD =\bigoplus_{i=1}^s m_i \, \kk[Z_i] \in \sd_\ba(M)$ with $\sdepth M = \sdepth \cD \ge t$, 
and $\cD_i \in \sd( (m_i \, \kk[Z_i])^{\le t})$ for each $1 \le i \le s$. 
Then the direct sum  $\bigoplus_{i=1}^s \cD_i$  gives a Stanley decomposition of $M^{\le t}$. 
Hence the assertion follows from Lemma \ref{sec:lem_sdep}. 
So it remains to  prove the ``if" part.  Assume that $\sdepth M^{\le t} \ge t$. We shall show that
$\sdepth M^{\le i} \ge t$ for all $i \ge t$ by induction on $i$. This implies the required assertion
since $M^{\le i} = M$ if $i \ge \dim_S M$.
If $i = t$, then there is nothing to do. Assume $i > t$.
Consider the exact sequence
$$
0 \longto M^{> i-1}/M^{>i} \longto M^{\le i} \longto M^{\le i-1} \longto 0.
$$
If $M^{>i-1}/M^{>i} = 0$, then $M^{\le i} = M^{\le i-1}$, and we are done.
Suppose not. By Lemma~\ref{sec:lem_of_skel}, we have $\sdepth \left(M^{>i-1}/M^{>i} \right) = i\ (\ge t)$. 
We also have $\sdepth(M^{\le i-1}) \geq t$ by the induction hypothesis. 
Therefore 
$$
\sdepth M^{\le i} \ge \min\mbra{\sdepth \left(M^{>i-1}/M^{>i} \right), \sdepth(M^{\le i-1})} \ge t.
$$
\end{proof}

\begin{thm}\label{skeleton CM}
The following are equivalent;
\begin{enumerate}
\item (Conjecture~\ref{Stanley conj}) $\sdepth M \ge \depth M$ for all $M \in \mod_{\ZZ^n} S$;
\item $\sdepth M \ge \depth M$ for all $M \in \mod_{\ZZ^n} S$ which are Cohen-Macaulay;
\item $\sreg (M) \ge \shreg (M)$ for all $M \in \mod_{\NN^n} S$;
\item $\sreg (M) \ge \shreg (M)$ for all $M \in \mod_{\NN^n} S$ with $\gs(M)= \sreg(M)$.
\end{enumerate}
\end{thm}

\begin{proof}
For (i) and (ii), we can replace $\mod_{\ZZ^n} S$ by $\mod_{\NN^n} S$. 
Hence the conditions (iii) and (iv) are the Alexander dual of (i) and (ii) respectively 
by Theorems \ref{sreg}, \ref{shreg} and the fact stated in the end of \S2. 

The implication (i) $\Rightarrow$ (ii) is clear.
For the converse implication, take $M \in \mod_{\NN^n} S$ with $t:=\depth M$.  Since $M \in \mod_\ba S$ 
for some $\ba \in \NN^n$, we can consider the skeleton $M^{\le t}$ of $M$. 
Since $M^{\le t}$ is Cohen-Macaulay and $\depth M^{\le t} = t$ 
as shown in Proposition \ref{sec:depth_and_skel}, the implication 
 (ii) $\Rightarrow$ (i) follows from Proposition~\ref{depth skeleton}. 
\end{proof}

\begin{rem}
(1) The equivalence (i) \iff (ii) is the module version of \cite[Corollary 2.2]{HSZ}.

(2) In the situation of  (ii), $\sdepth M \ge \depth M$ is equivalent to   
$\sdepth M = \depth M$ $(=\dim_S M)$. Similarly, in (ii), $\sreg (M) \ge \shreg (M)$ is equivalent to
$\shreg (M) = \sreg (M)$ ($=\gs(M)$). 

(3) We can replace $\mod_{\ZZ^n} S$ and $\mod_{\NN^n} S$ in the conditions of the 
theorem by $\mod_\ba S$ simultaneously. In particular, the above theorem 
holds in the context of squarefree modules. The equivalence (i) and (iii) has been mentioned in 
\cite{SoJ} for squarefree modules.  
\end{rem}

\section{Sliding operation for monomial ideals}
For $\ba, \bb \in \NN^n$, let $\ba \triangleleft \bb \in \NN^n$ be the vector whose $i^{\rm th}$ 
coordinate is 
$$(a \triangleleft b )_i = \begin{cases} a_i+b_i & \text{if $a_i \ne 0$,}\\
0 & \text{otherwise.}
\end{cases}$$   
Similarly, for $\ba, \bc \in \NN^n$ with $\ba \preceq \bc$, let $\bc \setminus \ba \in \NN^n$ 
denote the vector whose $i^{\rm th}$ coordinate is 
$$(c \setminus a )_i = \begin{cases} c_i+1-a_i & \text{if $a_i \ne 0$,}\\
0 & \text{otherwise.}
\end{cases}$$ 

Let $I \subset S$ be a monomial ideal minimally generated by $x^{\ba_1}, 
x^{\ba_2}, \ldots, x^{\ba_r}$, and $I= \bigcap_{i=1}^s \fm^{\bd_{i}}$ the irredundant irreducible 
decomposition. Here, for $\ba \in \NN^n$, $\fm^\ba$ denotes the irreducible ideal 
$( \, x_i^{a_i} \mid a_i > 0 \, )$.  For $\bb \in  \NN^n$, we set 
$$I^{\triangleleft \bb} := ( x^{\ba_1 \triangleleft \bb},  x^{\ba_2 \triangleleft \bb}, \ldots, 
 x^{\ba_r \triangleleft \bb}).$$
As we will see later, this operation preserves several invariants.   

Take $\bc \in \NN^n$ so that $\bc \succeq \ba_{i}$ for all $1 \leq i \leq r$. 
Then $I$ is positively $\bc$-determined, and we can take the Alexander dual $J:=\sA_\bc(S/I)$.  
By \cite[Theorems~5.24 and 5.27]{MS}, $J$ is (isomorphic to) a monomial ideal with  
$$J=(x^{\bc \setminus \bd_1}, x^{\bc \setminus \bd_{2}}, \ldots,x^{\bc \setminus \bd_{s}}) 
= \bigcap_{i=1}^r \fm^{\bc \setminus \ba_{i}}.$$ 
Similarly, $\sA_\bc(I) \cong S/J$. Hence we have the following. 

\begin{prop}
We have $I^{\triangleleft \bb} \cong \sA_{\bb+\bc} \circ \sA_{\bc}(I)$ and 
$S/I^{\triangleleft \bb} \cong \sA_{\bb+\bc} \circ \sA_{\bc}(S/I)$. 
Hence the irredundant irreducible decomposition of $I^{\triangleleft \bb}$ is given by 
$$I^{\triangleleft \bb}  = \bigcap_{i=1}^s \fm^{\bd_i \triangleleft \bb}.$$ 
\end{prop}

\begin{proof}
Since $(\bb+\bc) \setminus (\bc \setminus \ba) = \ba \triangleleft \bb$, the assertions 
easily follow from the above mentioned properties of the Alexander duality. 
\end{proof}

Through the inclusion $\mod_\bc S \embto \mod_{\bb+\bc} S$, we can consider the functor 
$$(-)^{\triangleleft \bb}:= \sA_{\bb+\bc} \circ \sA_{\bc}$$ from  
$\mod_\bc S$ to $\mod_{\bb+\bc} S$. 
Note that $S(-\ba)^{\triangleleft \bb} =S(-(\ba \triangleleft \bb))$ for $\ba\in \NN^n$.   
If $$\bigoplus_{i=1}^t S(-\ba'_i) \stackrel{\phi}{\too} \bigoplus_{i=1}^s S(-\ba_i) \too M \too 0$$
is the minimal presentation of $M \in \mod_\bc S$, then 
$$\bigoplus_{i=1}^t S(-(\ba'_i \triangleleft \bb)) \stackrel{\phi^{\triangleleft \bb}}{\too}
\bigoplus_{i=1}^s S(-(\ba_i \triangleleft \bb)) \too M^{\triangleleft \bb} \too 0$$
is the minimal presentation of $M^{\triangleleft \bb}$. 
Here, if $c x^\ba$ ($c \in \kk$ and $\ba \in \NN^n$) is an entry of the matrix 
representing $\phi$, then  $c x^{\ba \triangleleft \bb}$ 
is the corresponding entry of the matrix representing $\phi^{\triangleleft \bb}$.  
Hence $M^{\triangleleft \bb}$ does not depend on the particular choice of 
$\bc \in \NN^n$ with $M \in \mod_\bc S$, and 
we can regard $(-)^{\triangleleft \bb}$ as a functor from $\mod_{\NN^n} S$ to itself. 

\begin{prop}\label{triangle b basic}
For $M \in \mod_{\NN^n} S$ and $\bb \in \NN^n$, the following hold. 
$$  \beta_{i,\ba}(M)= \beta_{i, \ba \triangleleft \bb}(M^{\triangleleft \bb}) \ \ \text{(for all $i \in\NN$ and  $\ba \in \NN^n$)}, 
\qquad  \dim_S M = \dim_S M^{\triangleleft \bb}, $$  
$$\depth (M) = \depth (M^{\triangleleft \bb}), \qquad 
\sreg (M) = \sreg (M^{\triangleleft \bb}),$$ 
$$\sdepth (M) = \sdepth (M^{\triangleleft \bb}), \qquad 
\shreg (M) = \shreg (M^{\triangleleft \bb}).$$ 
\end{prop}

\begin{proof}
If  $P_\bullet$ is a minimal free resolution of $M$, then $(P_\bullet)^{\triangleleft \bb}$ 
is a minimal free resolution of $M^{\triangleleft \bb}$ by the exactness of the functor $(-)^{\triangleleft \bb}$.  
Since $P_i = \bigoplus_{\ba \in \NN^n} S(-\ba)^{\beta_{i, \ba}(M)}$, we have 
$(P_i)^{\triangleleft \bb} = \bigoplus_{\ba \in \NN^n} S(-(\ba \triangleleft \bb))^{\beta_{i, \ba}(M)}$. 
Hence $\beta_{i, \ba}(M)= \beta_{i, \ba \triangleleft \bb}(M^{\triangleleft \bb})$ holds, 
and this equation induces the third and fourth ones.  

For the remaining equations, take $\bc \in \NN^n$ with $M \in \mod_\bc S$. Then 
$$\dim_S M = n - \gs(\sA_\bc(M))= \dim_S(\sA_{\bb+\bc} \circ \sA_{\bc}(M)))
= \dim_S  M^{\triangleleft \bb}.$$
Similarly,  we have 
$$\sdepth (M) = n -\shreg(\sA_\bc(M))= \sdepth(\sA_{\bb+\bc} \circ \sA_{\bc}(M)) = \sdepth
 (M^{\triangleleft \bb}).$$ 
The equation $\shreg (M) = \shreg (M^{\triangleleft \bb})$ can be proved by the same way.
\end{proof}

The following is a direct consequence of Proposition~\ref{triangle b basic}. 

\begin{cor}\label{sliding preserves}
For $M \in \mod_{\NN^n} S$ and $\bb \in \NN^n$, we have the following. 

(1) $M$ is Cohen-Macaulay if and only if so is $M^{\triangleleft \bb}$. 
Similarly, for a monomial ideal $I$, $S/I$ is Gorenstein if and only if so is $S/I^{\triangleleft \bb}$. 

(2) Stanley's conjecture  holds for $M$ if and only if it holds for $M^{\triangleleft \bb}$.  
\end{cor}

Unfortunately (?), many classes of monomial ideals for which Stanley's conjecture  has been proved 
is closed under the operation $(-)^{\triangleleft \bb}$. 
For example, a monomial ideal $I$ is Borel fixed if and only if so is $I^{\triangleleft \bb}$. 
Hence Corollary~\ref{sliding preserves} does not so much widen the region where the conjecture holds. 
The following is an exception.  

Let $I$ be a monomial ideal minimally generated by monomials $m_1, \ldots, m_r$.  
We say $I$ has {\it linear quotient} if after suitable change  of the order of $m_i$'s  
the colon ideal $(m_1, \ldots, m_{i-1}):m_i$ is a monomial prime ideal for all $2 \leq i \leq r$.   
For example, $I:=(xy, yz^2) \subset \kk[x,y,z]$ has linear quotient, but 
$I^{\triangleleft (1,0,0)}=(x^2y, yz^2)$ does not. 
For further information on this notion, consult \cite{HVZ} and references cited there. 
Here we just remark that, for squarefree monomial ideals, 
having linear quotient is the Alexander dual notion of (non-pure) shellability,   
and  there are many examples. 

Since Stanley's conjecture holds for a monomial ideal with linear quotient by 
\cite[Proposition~4.5]{SoJ}, we have the following. 

\begin{prop}
If a monomial ideal $I$ has linear quotient then Stanley's conjecture holds for 
$I^{\triangleleft \bb}$ for all $\bb \in \NN^n$. 
\end{prop}

\begin{rem}
Let $I$ be a complete intersection monomial ideal of codimension $c$. Then each variable $x_i$ appears in 
at most one minimal monomial generator of $I$. Hence there is $\bb \in \NN^n$ 
such that  $(\sqrt{I})^{\triangleleft \bb} = I$ and we have $\sdepth \sqrt{I} = \sdepth I$ 
by Proposition~\ref{triangle b basic}. The latter equation has been proved by Cimpoea\c s \cite{C}. 
Now it is known that  $\sdepth I = n - \lfloor \frac{c}{2} \rfloor$ by Shen \cite{She}, 
but the equation $\sdepth \sqrt{I} = \sdepth I$ is used in his proof. 
\end{rem}

\section{Quotient ring by a cogeneric monomial ideal}
\begin{dfn}[Bayer et al. \cite{BPS}]\label{generic def}
Let $I$ be a monomial ideal minimally generated by monomials $m_1, \ldots, m_r$. 
We say $I$ is {\it generic} if any distinct $m_i$ and $m_j$ do not have the same 
non-zero exponent in any variable. 
\end{dfn}

\begin{dfn}[Sturmfels \cite{Str}]
Let $I$ be a monomial ideal with the irredundant irreducible decomposition 
$I=\bigcap_{i=1}^s \fm^{\ba_i}$. 
We say $I$ is {\it cogeneric} if any distinct $\fm^{\ba_i}$ and $\fm^{\ba_j}$ do not have 
the same minimal (monomial) generator. 
\end{dfn}

\begin{rem} 
(1) It is easy to see that a monomial ideal $I$ is generic 
if and only if the Alexander dual $J=\sA(S/I)$ is cogeneric.  Similarly, for $\bb \in \NN^n$, 
$I$ is generic (resp. cogeneric) if and only if so is $I^{\triangleleft \bb}$. 

(2) In \cite{MSY}, more inclusive definitions of generic and cogeneric monomial 
ideals are given,  and Apel \cite{A1, A2} uses these definitions.      
However, our proof of Theorem~\ref{cogeneric} below only works for the original definition, 
that is, Stanley's conjecture  is still open for the quotients by (non-Cohen-Macaulay) cogeneric 
monomial ideals in the sense of \cite{MSY}.  
\end{rem}

\begin{thm}[{Apel~\cite[Theorem~5]{A2}}]\label{Apel cogeneric} 
If $I$ is a Cohen-Macaulay cogeneric monomial ideal, 
then Stanley's conjecture holds for $S/I$ (i.e., $\sdepth(S/I) = \depth(S/I)$ holds, in this case). 
\end{thm}

The next result says that the Cohen-Macaulay 
assumption can be removed from the above theorem.  

\begin{thm}\label{cogeneric}
If $I$ is a cogeneric monomial ideal, then $\sdepth(S/I) \geq \depth(S/I)$. 
That is, Stanley's conjecture  holds for the quotient by a cogeneric monomial ideal. 
\end{thm}

Let $I$ be a monomial ideal and $J:=\sA(S/I)$ the Alexander dual. 
As stated in the end of \S2,
 $S/I$ is Cohen-Macaulay if and only if $\sreg(J)=\gs(J)$, 
where $\gs(J)= \min \{ \, \# \supp(\ba) \mid x^\ba \in J \, \}$.  

The next result is just the Alexander dual of Theorem~\ref{Apel cogeneric}.   

\begin{prop}\label{Apel generic}
Let $I$ be a generic monomial ideal with $\sreg(I)=\gs(I)$. 
Then we have $\shreg(I) = \sreg(I)$. 
\end{prop}

Via the Alexander duality, Theorem~\ref{cogeneric} is equivalent to the next. 
This is just a ``direct translation". 
However, it improves the ``human interface" of the argument, 
since we usually describe ideals by their generators, not irreducible decompositions.  
Anyway, to prove Theorem~\ref{cogeneric}, it suffices to show 
Theorem~\ref{generic} below.   
 
\begin{thm}\label{generic}
If $I$ is a generic monomial ideal,  then $\shreg(I) \leq \sreg(I)$. 
\end{thm}

\noindent{\it Proof.}
We prove the assertion by backward induction on $\gs(I)$. 
If $\gs(I)=n$, then $\shreg(I) = \sreg(I) =n$ and the 
assertion holds. Consider the case when $s:=\gs(I) <n$. 

Let $m_1, \ldots, m_r$ be the minimal monomial  generators of $I$. 
Replacing $I$ by $I^{\triangleleft \br}$ for 
$\br = (r,r, \ldots, r) \in \NN^n$, we may assume that 
we have $a_i > r$ for all $x^\ba \in I$ with $a_i \ne 0$. 
Assume that $\# \supp(m_i) = s$ for all $1 \leq i \leq t$ 
and $\# \supp(m_i) > s$ for all $i > t$. Consider the monomial ideals 
$$I_i = ( \, x_j^i \cdot m_i \mid j \not \in \supp(m_i) \,)$$ 
for each  $1 \leq i \leq t$, and set 
$$J := I_1 + I_2+\cdots + I_t+ (m_{t+1}, \cdots, m_r).$$   
Then $J$ is a generic monomial ideal with $J \subset I$ 
and $\gs(J) = s+1$. Moreover, we have the following lemma  
whose proof will be given later.

\begin{lem}\label{I/J} With the above notation, we have
$$\sreg(I/J) = \shreg(I/J)=s.$$
\end{lem}

\noindent{\it The continuation of the proof of Theorem~\ref{generic}.}
We have the short exact sequence 
$$0 \to J \to I \to I/J \to 0.$$
By Lemma~\ref{I/J}, Remark~\ref{suupp reg basic} and the fact that $\sreg(J) \geq s+1$, 
we have $\sreg(J) = \sreg(I)$ unless $\sreg(I) = s$. 
If $\sreg(I) =s$, then $\shreg(I)=s$ by Proposition~\ref{Apel generic}. 
Therefore we may assume that $\sreg(J) = \sreg(I)$. 
By the induction hypothesis, $\shreg(J) \leq \sreg(J)$. 
Hence we have 
\begin{align*}
\shreg(I) \leq \max \{ \shreg(J), \shreg(I/J) \} 
&= \shreg(J) \\
&\leq \sreg(J)=\sreg(I).
\end{align*}
\qed

\noindent{\it Proof of Lemma~\ref{I/J}.} Set $M:= I/J$, and consider $\shreg M$ first. 
It is clear that $\shreg M \geq s$, 
and it suffices to show that $\shreg M \leq s$. 

If $M_\ba \ne 0$, then $\# \supp^\br(\ba) =s$. 
For a subset $F \subset [n]:=\{ 1, \ldots, n \}$ with $\# F = s$, set 
$$M_{[F]}:= \bigoplus_{\substack{\ba \in \NN^n \\ \supp^\br(\ba) =F}} M_\ba.$$
Then it is an $S$-submodule of $M$, and we have 
\begin{equation}\label{direct sum2}
M= \bigoplus_{\substack{F \subset[n] \\ \# F = s}} M_{[F]}
\end{equation}
as $S$-modules. So it suffices to show that $\shreg(M_{[F]}) \leq s$ 
for each $F \subset [n]$ with $\#F =s$. We may assume that $M=M_{[F]}$ 
and $I=(m_1, \ldots, m_r)$ with $\supp(m_i)=F$ for all $i$  
(this reduction slightly restricts the structure of the module $M_{[F]}$, 
but it causes no problem in the following argument). 

Set $\ba := \deg(m_1) \vee \deg(m_2) \vee \cdots \vee \deg(m_r)$.  
By the assumption that $\supp(m_i)=F$ for all $i$, we have $\supp(\ba)=F$. 
Note that the $i^{\rm th}$ coordinate of $\ba \vee \br$ is $a_i$ if $i \in F$, and 
$r$ if $i \not \in F$. 
Hence $I$, $J$ and $M$ are positively $(\ba \vee \br)$-determined. 
We will give a decomposition $\cD \in \qsd_{\ba \vee \br} (M)$ with $\shreg \cD = s$. 
Set $\Sigma := \{ \, \bb \in \NN^n \mid x^\bb \in I, \, \bb \preceq \ba \, \},$
and take $\bb \in \Sigma$.  
Since $\supp(\bb)= \supp(\ba)=F$, we have $x^\bb \not \in J$. 
Moreover, for all monomial $x^\bc$ with $\supp(\bc) \subset \supp^\ba(\bb)$ and all 
$j \not \in F$, we have 
$$\min \{ \, i  \mid  \text{$m_i$ divides  $x^\bb$} \, \} 
= \min \{ \,  i \mid \, (x_j)^i \cdot  x^{\bb+\bc} \in J \, \}=:l(\bb)$$
by the construction of $J$.  Let $\bb' \in \NN^n$ be the vector whose $i^{\rm th}$ 
coordinate is 
$$
b'_i=\begin{cases}
b_i & \text{if $i \in F$,}\\
l(\bb)-1 & \text{if $i \not \in F$.}
\end{cases}
$$
Then 
\begin{equation}\label{qsd I/J}
\cD:=\bigoplus_{\bb \in \Sigma} \kk_{\ba \vee \br}[\bb, \bb']
\end{equation}
is a quasi Stanley decomposition of $M$ with $\shreg \cD = s$. 

To compute $\sreg (M)$, we can use the direct sum \eqref{direct sum2}, and 
may assume that $\supp(m_i)=F$ for all $i$ again.    
To prove $\sreg (M) = s$, we show that the quasi Stanley decomposition \eqref{qsd I/J} 
induces a filtration of $M$ as an $S$-module.  
Note that $\ba$ is the largest element of $\Sigma$ with respect to the order $\succeq$. 
Set $\bb_1:= \ba$, and take  a maximal element $\bb_2$ of 
$\Sigma \setminus \{ \bb_1 \}$.  
Inductively, let $\bb_i$ be a  maximal element of $\Sigma \setminus \{ \bb_1, \ldots, \bb_{i-1} \}$. 
This procedure stops in finite steps, since $m:=\# \Sigma < \infty$.  For $i \geq 1$, 
let $M_i$ denote the quotient module of $M$ by the submodule generated by the images of the monomials 
$x^{\bb_1}, \ldots, x^{\bb_i}$ (set $M_0:=M$), and let $N_i$ be the submodule of $M_{i-1}$ generated by the image of the 
monomial $x^{\bb_i}$.  Then we have the short exact sequence 
 $$0 \to N_i \to M_{i-1} \to M_i \to 0$$ 
in $\mod_{\NN^n} S$ for each $1 \leq i \leq m$.  Moreover, we have 
$$N_i \cong \kk_{\ba \vee \br}[\bb_i, \bb'_i] \quad \text{and } \quad M_m =0.$$
Since $\sreg (N_i) = s$ for all $i$ (see the comment before Definition~\ref{shreg dfn}), 
we can proved that $\sreg (M_i) =s$ for all $i$ by backward induction on $i$ 
starting from $i=m-1$. Since $M=M_0$, we are done. 
\qed 

\section*{Acknowledgement}
The authors are grateful to an anonymous referee for pointing out gaps in an earlier 
version of the paper.

\end{document}